\documentclass[a4paper,10pt]{article}
\usepackage[utf8x]{inputenc}
\pdfoutput=1

\usepackage{times}
\usepackage{geometry}
\usepackage[T1]{fontenc}
\usepackage{color}
\usepackage{amsmath}
\usepackage{amssymb}
\usepackage{array}
\usepackage{amsthm}
\usepackage{graphicx}
\usepackage{hyperref}
\usepackage{setspace}
\usepackage{stmaryrd}

\theoremstyle{plain}
\newtheorem{thm}{Theorem}[section]
\newtheorem{prop}[thm]{Proposition}

\newtheorem{lem}[thm]{Lemma}

\theoremstyle{definition}
\newtheorem{defn}[thm]{Definition}

\newtheorem{exmp}[thm]{Example}

\theoremstyle{remark}
\newtheorem{rem}[thm]{Remark}

\theoremstyle{plain}

\newcommand{\R}{\mathbb{R}}

\newcommand{\C}{\mathbb{C}}

\newcommand{\N}{\mathbb{N}}

\newcommand{\CP}{\mathbb{C}P}

\newcommand{\dimn}{\mathrm{dim}}

\newcommand{\identity}{\mathrm{id}}

\newcommand{\scal}{\mathrm{scal}}
\newcommand{\ric}{\mathrm{Ric}}
\newcommand{\pr}{\mathrm{pr}}
\newcommand{\trace}{\mathrm{tr}}
\newcommand{\kernel}{\mathrm{ker}}

\newcommand{\dv}{\text{ }dV}
\newcommand{\edv}{\text{ }e^{-f}dV}
\newcommand{\efgdv}{\text{ }e^{-f_g}dV}

\newcommand{\Diff}{\mathrm{Diff}}

\newcommand{\spectrum}{\mathrm{spec}}
\newcommand{\gradient}{\mathrm{grad}}

\newcommand*{\suchthat}[1]{\left|\vphantom{#1}\right.}

\newcommand{\Sol}{\mathrm{Sol}}
\newcommand{\Iso}{\mathrm{Iso}}

\renewcommand{\title}[1]{{\bfseries #1}\par}
\renewcommand{\author}[1]{\medskip{#1}\par\smallskip}
\newcommand{\affiliation}[1]{{\itshape #1}\par}
\newcommand{\email}[1]{E-mail:~\texttt{#1}\par}

\numberwithin{equation}{section}

\begin{document}
\begin{center}
\title{\LARGE Rigidity and infinitesimal deformability of Ricci solitons}
\vspace{3mm}
\author{\Large Klaus Kröncke}
\vspace{3mm}
\affiliation{Universität Regensburg, Fakultät für Mathematik\\Universitätsstraße 31\\93053 Regensburg, Germany}
\vspace{3mm} 
\affiliation{Universität Potsdam, Institut für Mathematik\\Am Neuen Palais 10\\14469 Potsdam, Germany} 
\email{klaus.kroencke@mathematik.uni-regensburg.de} 
\end{center}
\vspace{2mm}
\begin{abstract}In this paper, an obstruction against the integrability of certain infinitesimal solitonic deformations is given. Using this obstruction, we show that the complex projective spaces of even complex dimension are rigid as Ricci solitons although they have infinitesimal solitonic
deformations.
\end{abstract}
\section{Introduction}
 A Riemannian manifold $(M,g)$ is called a Ricci soliton if there exist a vector field $X\in\mathfrak{X}(M)$ and a constant $c\in\R$ such that the Ricci tensor satisfies the equation
\begin{align}\label{solitonequation}
 \ric_g+\frac{1}{2}L_Xg=c\cdot g.
\end{align}
Ricci solitons were first introduced by Hamilton in the eighties \cite{Ham88}. They appear as self-similar solutions of the Ricci flow: Ricci solitons evolve particularly simple under the Ricci flow, namely by diffeomorphisms and rescalings.

A soliton is called gradient, if $X=\gradient f$ for some $f\in C^{\infty}(M)$. We call $(M,g)$ expanding, if $c<0$, steady, if $c=0$ and shrinking, if $c>0$.
 If $X=0$, we recover the definition of an Einstein metric with Einstein constant $c$. If $X\neq0$, we call the soliton nontrivial.
In the compact case, any soliton is gradient and any expanding or steady soliton is Einstein. 
For a detailed expository to the theory of Ricci solitons, see e.g.\ \cite{Cao10,CC07}.

In this paper, we study the moduli space of Ricci solitons (the set of Ricci solitons modulo diffeomorphism and rescaling) on compact manifolds.
A local description of this moduli space can be given using a modified version of Ebin's slice theorem provided by Podest{\`a} and Spiro \cite{PS13}.
They construct a slice to the natural action of the diffeomorphism group on the space of metrics with tangent space $\kernel(\delta_f)$. Here, $\delta_f$ is the divergence
weighted with the smooth function $f$.

Let $(M,g)$ be a Ricci soliton. A tensor $h\in C^{\infty}(S^2M)$ is called infinitesimal solitonic deformation, if $\delta_f(h)=0$ (where $f$ is the soliton potential), $\int_M \trace h\edv=0$
and $h$ lies in the kernel of the linearization of \eqref{solitonequation}.
If $g$ is not an isolated point in the moduli space, it admits infinitesimal solitonic deformations.
 Conversely, given an infinitesimal solitonic deformation $h$, it is not clear whether it is integrable, i.e.\ if there exists a curve of Ricci solitons tangent to $h$.
 In this paper, we show that for certain deformations, this is not the case. Moreover, we prove that $\CP^{2n}$ with the Fubini-Study metric is an isolated point in the moduli space of Ricci solitons although it has solitonic deformations.

Some obstructions against the existence of infinitesimal solitonic deformations are given in \cite{PS13}.
Analoguous questions have been studied in the Einstein case before \cite{Koi78,Koi80,Koi82,Koi83}, see also \cite{Bes08} and the methods used here are quite similar. 

This paper is organized as follows: In Section \ref{shrinkerentropy}, we introduce Perelman's shrinker entropy, which provides a variational characterization of shrinking Ricci solitons as its critical points.
In Section \ref{modulispace}, we introduce various notions of rigidity. We mention a theorem of \cite{PS13} describing the structure of the moduli space of Ricci solitons and generalizing
previous results obtained in the Einstein case \cite{Koi83}.
We prove a criterion for weak solitonic rigidity of Einstein manifolds:
\begin{prop}\label{weakrigidity}
 Let $(M,g)$ be an Einstein manifold with Einstein constant $\mu>0$. If $2\mu\notin\spectrum(\Delta)$, any Ricci soliton $H^s$-close (with $s\geq[\frac{n}{2}]+3$) to $g$ is Einstein.
\end{prop}
In Section \ref{nonintegrability}, we discuss integrability of infinitesimal solitonic deformations. 
If $(M,g)$ is Einstein with constant $\mu$ and $2\mu$ is a positive eigenvalue of the Laplacian, we have infinitesimal solitonic deformations which can be formed of the corresponding eigenfunctions.
For these deformations, we prove an obstruction against integrability:
\begin{thm}\label{rigiditytheorem}
 Let $(M,g)$ be an Einstein manifold with Einstein constant $\mu>0$. Let $v\in C^{\infty}(M)$ be such that $\Delta v=2\mu v$. Then the infinitesimal solitonic deformation $\mu v\cdot g+\nabla^2 v$ is not integrable if there exists another function
 $w\in C^{\infty}(M)$ with $\Delta w=2\mu w$ such that
\begin{align*}
 \int_M v^2 w \dv\neq 0.
\end{align*}
\end{thm}
\noindent
Concrete examples are discussed in Section \ref{examples}. We use the above theorem to prove
\begin{thm}\label{CP2n}
 The $\CP^{2n}$ with the Fubini-Study metric is isolated in the moduli space of Ricci solitons although it has infinitesimal solitonic deformations. 
\end{thm}
\section{Notation and conventions}
 Throughout, any manifold $M$ will be compact and any metric considered on $M$ will be smooth. The dimension of $M$ will be $n$.
For the Riemann curvature tensor, we use the sign convention such that
 $R_{X,Y}Z=\nabla^2_{X,Y}Z-\nabla^2_{Y,X}Z$. Given a fixed metric, we equip the bundle of $(r,s)$-tensor fields (and any subbundle) with the natural pointwise scalar product induced by the metric.
By $S^pM$, we denote the bundle of symmetric $(0,p)$-tensors.
Let $f\in C^{\infty}(M)$. We introduce some differential operators weighted by $f$. The $f$-weighted Laplacian (or Bakry-Emery Laplacian) acting on $C^{\infty}(S^pM)$ is
\begin{align*}
 \Delta_f h=-\sum_{i=1}^n \nabla^2_{e_i,e_i}h+\nabla^2_{\gradient f}h.
\end{align*}
By the sign convention, $\Delta_f=(\nabla^*_f)\nabla$, where $\nabla^*_f$ is the adjoint of $\delta$ with respect to the weighted $L^2$-scalar product $\int_M\langle.,.\rangle \edv$.
The weighted divergence $\delta_f:C^{\infty}(S^pM)\to C^{\infty}(S^{p-1}M)$ and its formal adjoint $\delta_f^*\colon C^{\infty}(S^{p-1}M)\to C^{\infty}(S^pM)$ with respect to the weighted scalar product are given by
\begin{align*}\delta_f T(X_1,\ldots,X_{p-1})=&-\sum_{i=1}^n\nabla_{e_i}T(e_i,X_1,\ldots,X_{p-1})+T(\gradient f,X_1,\ldots,X_{p-1}),\\
            \delta_f^*T(X_1,\ldots,X_p)=&\frac{1}{p}\sum_{i=0}^{p-1}\nabla_{X_{1+i}}T(X_{2+i},\ldots,X_{p+i}),
\end{align*}
where the sums $1+i,\ldots,p+i$ are taken modulo $p$. If $f$ is constant, we recover the usual notions of Laplacian and divergence. In this case, we will drop the $f$ in the notation.
For $\omega\in\Omega^1(M)$, we have $\delta_f^*\omega=\frac{1}{2}L_{\omega^{\sharp}}g$ where $\omega^{\sharp}$ is the vector field dual to $\omega$ and $L$ is the Lie derivative.
Thus, $\delta_f^*(\Omega^1(M))$ is the tangent space of the manifold $g\cdot \Diff(M)=\left\{\varphi^*g|\varphi\in\Diff(M)\right\}$.
   \section{The shrinker entropy}\label{shrinkerentropy}
Let $g$ be a Riemannian metric, $f\in C^{\infty}(M)$, $\tau>0$ and define
\begin{align*}\mathcal{W}(g,f,\tau)=\frac{1}{(4\pi\tau)^{n/2}}\int_M [\tau(|\nabla f|^2_g+\scal_g)+f-n]\edv,
\end{align*}
where $\scal_g$ is the scalar curvature of $g$.
Let\index{$\mu_(g,\tau)$}
 \begin{align*}\mu(g,\tau)&=\inf \left\{\mathcal{W}(g,f,\tau)\suchthat{f\in C^{\infty}(M),\frac{1}{(4\pi\tau)^{n/2}}\int_M \edv_g=1}f\in C^{\infty}(M),\frac{1}{(4\pi\tau)^{n/2}}\int_M \edv_g=1\right\}.
 \end{align*}
For any fixed $\tau>0$, the infimum is finite and is realized by a smooth function \cite[Lemma 6.23 and 6.24]{CC07}.
We define the shrinker entropy\index{shrinker entropy} as
\begin{align*}\nu(g)&=\inf \left\{\mu(g,\tau)\mid\tau>0\right\}.
\end{align*}
This functional was first introduced in the pioneering work of Perelman \cite{Per02}.
If the smallest eigenvalue of the operator $4\Delta_g+\scal_g$ is positive,
 $\nu(g)$ is finite and realized by some $\tau_g>0$ \cite[Corollary 6.34]{CC07}. By construction, $\nu$ is scale and diffeomorphism invariant.
Its first variation is 
\begin{align*}\nu(g)'(h)=-\frac{1}{(4\pi\tau_g)^{n/2}}\int_M\left\langle \tau_g(\ric+\nabla^2 f_g)-\frac{1}{2}g,h\right\rangle\text{ }e^{-f_g}dV_g,
\end{align*}
where $(f_g,\tau_g)\in C^{\infty}(M)\times \R_+$ is a pair realizing $\nu(g)$ (see e.g.\ \cite[Lemma 2.2]{CZ12}).
The critical points of $\nu$ are precisely the shrinking Ricci solitons. By the above, these are the metrics for which we have the equation
\begin{align}\label{Riccisoliton}
 \ric+\nabla^2 f_g=\frac{1}{2\tau_g}g.
\end{align}
For any Ricci soliton $g$, there exists a $C^{2,\alpha}$-neighbourhood $\mathcal{U}$ in the space of metrics on which $\nu$ depends analytically on the metric. Moreover,
the minimizers $f_g,\tau_g$ are unique on $\mathcal{U}$ and depend analytically on the metric \cite[Lemma 4.1]{Kro14}. In the particular case of an Einstein metric, $f_g$ is constant and $\tau_g=\frac{1}{2\mu}$ where $\mu$ is the Einstein constant.
\begin{prop}[{\cite{CHI04,CZ12}}]\label{secondnu}Let $(M,g)$ be a shrinking Ricci soliton. Then the second variation of $\nu$ at $g$ is given by
 \begin{align}
  \nu''_{g}(h)=\frac{\tau}{(4\pi\tau)^{n/2}}\int_M \langle N h,h\rangle \edv,
 \end{align}
where $(f,\tau)$ is the minimizing pair realizing $\nu$. The stability operator $N$ is given by
\begin{align}
 Nh=-\frac{1}{2}\Delta_f h+\mathring{R}h+\delta^*_f(\delta_f(h))+\frac{1}{2}\nabla^2 v_h-\ric \frac{\int_M \langle \ric,h\rangle \edv}{\int_M \scal \edv}.
\end{align}
Here, $\mathring{R}h(X,Y)=\sum_{i=1}^nh(R_{e_i,X}Y,e_i)$ and $v_h$ is the unique solution of
\begin{align}\label{v_h}
 (-\Delta_f+\frac{1}{2\tau})v_h=\delta_f(\delta_f(h)).
\end{align}
\end{prop}
\begin{rem}
 The operator $N:C^{\infty}(S^2M)\to C^{\infty}(S^2M)$ is formally self-adjoint with respect to $L^2(e^{-f}dV)$. Since $\nu$ is scale and diffeomorphism invariant, $N$ vanishes on
\begin{align*}
 W:=\R\cdot \ric\oplus \delta_f^*(\Omega^1(M)).
\end{align*}

\end{rem}

\section{The moduli space of Ricci solitons}\label{modulispace}
Let $\mathcal{M}$ be the set of smooth metrics on $M$ and assume that $s\in\N$ satisfies $s\geq [\frac{n}{2}]+3$.
According to \cite{PS13}, we introduce the following notions:
\begin{defn}
 A shrinking Ricci soliton $g$ is called rigid, if there exists a $H^s$-neighbourhood $\mathcal{U}\subset\mathcal{M}$ such that any shrinking Ricci soliton $\tilde{g}\in\mathcal{U}$ is homothetic to $g$, i.e.\ there exist $\lambda>0$ and $\varphi\in\Diff(M)$ such that
$g=\lambda\varphi^*\tilde{g}$. Analoguously, an Einstein metric $g$ is called rigid, if there exists a  $H^s$-neighbourhood $\mathcal{U}\subset\mathcal{M}$ such that any Einstein metric in $\mathcal{U}$ is homothetic to $g$.
An Einstein metric is called weakly solitonic rigid, if there exists a  $H^s$-neighbourhood $\mathcal{U}\subset\mathcal{M}$ such that any Ricci soliton in $\mathcal{U}$ is Einstein.
\end{defn}
\noindent
Let $f\in C^{\infty}(M)$ and $g\in\mathcal{M}$. In \cite{PS13}, Podest{\`a} and Spiro construct a set $\mathcal{S}^s_{f}\subset\mathcal{M}$ satisfying the following properties:
\begin{itemize}
 \item There exists a small $H^s$-neighbourhood $\mathcal{U}$ of $g$ in the set of metrics such that any $\tilde{g}\in\mathcal{U}$ is isometric to a unique metric $\hat{g}\in\mathcal{S}^s_{f}$.
 \item $\mathcal{S}^s_{f}$ is a smooth manifold with tangent space $T_{g}\mathcal{S}^s_{f}=\kernel(\delta_{g,f})$.
\end{itemize}
Such a set is called an $f$-twisted slice. If $f\equiv0$, we recover the slice constructed in Ebin's slice theorem \cite[Theorem 7.1]{Eb70}.

Now the question, whether all Ricci solitons in a neighbourhood of a given Ricci soliton $g$ are homothetic to $g$ reduces to the question whether $g$ is an isolated point in
\begin{align*}
 \mathcal{S}ol^s_{g}=\left\{\tilde{g}\in\mathcal{S}^s_{f_g}\text{ }\bigg\vert \text{ }\ric_{\tilde{g}}+\nabla^2f_{\tilde{g}}=\frac{1}{2\tau_g}\cdot \tilde{g}\right\}.
\end{align*}
Note that we fixed the constant $\tau_{g}$ in order to avoid rescalings of metrics.
Let $g_t$ a $C^1$-curve in $\mathcal{S}ol^s_g$. Then
 we have \begin{align*}
          \frac{d}{dt}\bigg\vert_{t=0}\tilde{g}_t=h\in  V:=\left\{h\in C^{\infty}(S^2M)\text{ }\bigg\vert \text{ }\delta_f(h)=0 \text{ and }\int_M \langle\ric,h\rangle \edv=0\right\}
\end{align*}
 and $\nu''(h)=0$ because $\nu$ is constant along $\tilde{g}_t$.
The space $V$ is the $L^2(e^{-f}dV)$-orthogonal complement of the space $W$ defined above.
 This motivates the following definition:
\begin{defn}\label{DefnISD}
 Let $(M,g)$ be a gradient shrinking Ricci soliton and let $N$ be the stability operator of Proposition \ref{secondnu}. We call $h\in C^{\infty}(S^2M)$ an infinitesimal solitonic deformation if $h\in V$ and $N(h)=0$. An infinitesimal solitonic deformation is called integrable if there exists a $C^1$-curve of Ricci solitons $g_t$ through $g=g_0$
such that $\frac{d}{dt}|_{t=0}g_t=h$.
\end{defn}
In general, one cannot expect that $ \mathcal{S}ol^s_g$ is a manifold, but the following holds \cite[Theorem 3.4]{PS13}: There exists an analytic finite-dimensional submanifold $\mathcal{Z}^s\subset\mathcal{S}^s_{f_g}$
such that $T_g\mathcal{Z}^s=\kernel{N|_V}$ and $ \mathcal{S}ol^s_g$ is an analytic subset of $\mathcal{Z}^s$.
If all infinitesimal solitonic deformations are integrable, we have $ \mathcal{S}ol^s_g=\mathcal{Z}^s$ (possibly after passing to smaller neighbourhoods).

\begin{defn}[{\cite[p.\ 347]{Bes08}}]Let $(M,g)$ be an Einstein manifold. We call $h\in C^{\infty}(S^2M)$ an infinitesimal Einstein deformation if $\delta h=0$, $\trace h=0$ (i.e.\ $h$ is a transverse traceless tensor) and $\nabla^*\nabla h-2\mathring{R}h=0$. An infinitesimal Einstein deformation is called integrable if there exists a $C^1$-curve of Einstein metrics $g_t$ through $g=g_0$
such that $\frac{d}{dt}|_{t=0}g_t=h$.
\end{defn}
The set of Einstein metrics close to $g$ admits similar properties as the set of Ricci solitons \cite[Theorem 3.1]{Koi83}.
Suppose that $g$ is an Einstein manifold with constant $\mu$.
 Let $\mathcal{S}^s$ be the slice in the space of metrics as constructed in Ebin's slice theorem. Then, there exists an analytic finite-dimensional manifold
$\mathcal{W}^s\subset \mathcal{S}^s$ such that $T_g\mathcal{W}^s=\kernel((\nabla^*\nabla-2\mathring{R})|_{TT})$ and
\begin{align*}
 \mathcal{E}_g^s=\left\{\tilde{g}\in\mathcal{S}^s|\ric_{\tilde{g}}=\mu\tilde{g}\right\}
\end{align*}
is an analytic subset of $\mathcal{W}^s$.
Let $IED$ be the space of infinitesimal Einstein deformations and $ISD$ be the space of infinitesimal solitonic deformations. Then we have
\begin{align*}
 ISD=IED\oplus \left\{\mu v\cdot g+\nabla^2 v|v\in C^{\infty}(M),\Delta v=2\mu v\right\},
\end{align*}
see \cite[Lemma 6.2]{Kro14}. Now we prove the statement about weak solitonic rigidity stated in the introduction:
\begin{proof}[Proof of Proposition \ref{weakrigidity}]
 Suppose that $g$ is not weakly solitonic rigid. Then we have a sequence of nontrivial Ricci solitons $g_i\to g$ in $H^s$. For nontrivial Ricci solitons, it is well-known that $\frac{1}{\tau_{g_i}}\in\spectrum(\Delta_{f_{g_i}})$ and an eigenfunction is given by $f_{g_i}-\frac{n}{2}-\nu(g_i)$.
 This follows from the Euler-Lagrange equation for $f_{g_i}$ \cite[p.\ 5]{CZ12}.
For the $k$th eigenvalue of $\Delta_{f_g}$, we have the minimax characterization
\begin{align*}
 \lambda_k(\Delta_{f_g})=\min_{\substack{U\subset C^{\infty}(M)\\ \dimn{U}=k}}\max_{\substack{u\in U\\u\neq0}}\frac{\int_M |\nabla u|^2 \efgdv_g}{\int_M u^2 \efgdv_g}
\end{align*}
which shows that the spectrum of $\Delta_{f_g}$ depends continuously on the metric with respect to the $C^{2,\alpha}$-topology.
By Sobolev embedding, $H^s$-continuity follows.
 This implies that $\frac{1}{\tau_g}=2\mu\in\spectrum(\Delta)$ and the proof is finished by contradiction.
\end{proof}
\begin{rem}
 This proposition generalizes \cite[Proposition 5.3]{PS13}, where only the case of K\"ahler-Einstein manifolds is considered.
\end{rem}
\section{Non-integrability of solitonic deformations}\label{nonintegrability}
In this section we use similar arguments as in \cite{Koi82}, where the integrability of infinitesimal Einstein deformations was discussed. 
Consider the tensor
\begin{align*}
 \Sol_g=\tau_g(\ric_g+\nabla^2 f_g)-\frac{g}{2}.
\end{align*}
Obviously, the zero set of the map $ g\mapsto \Sol_g$ equals the set of shrinking Ricci solitons. This map is well-defined and analytic in an open $C^{2,\alpha}$-neighbourhood of a given Ricci soliton.
Suppose now, we have a smooth curve of Ricci solitons $g_t$. From differentiating the equation $\Sol=0$ along $g_t$ twice, we have
\begin{align*}
 D\Sol(g^{(1)})=0,\qquad D\Sol(g^{(2)})+D^2\Sol(g^{(1)},g^{(1)})=0.
\end{align*}
Here, $D^k\Sol$ denotes the k'th Fr\'{e}chet derivative of the map $g\mapsto\Sol_g$ and $g^{(k)}$ denotes the k'th derivative of the curve $g_t$. More generally, from differentiating $k$ times, we obtain
\begin{align}\label{recursion0}
 D\Sol(g^{(k)})+\sum_{l=2}^{k}\sum_{\substack{ 1\leq k_1\leq\ldots\leq k_l\\ k_1+\ldots+k_l=k}} C(k,l,k_1,\ldots,k_l) D^l\Sol(g^{(k_1)},\ldots,g^{(k_l)})=0
\end{align}
where $C(k,l,k_1,\ldots,k_l)\in\N$ are constants only depending on $k,l,k_1,\ldots,k_l$.
Let now $g$ be a Ricci soliton and $h$ an infinitesimal solitonic deformation. Suppose that $h$ is integrable.
By projecting to the slice and rescaling, we may assume that there is a curve $g_t$ in the set $\mathcal{S}ol^s_g$ such that $g_0^{(1)}=h$.
 By analyticity of the set $\mathcal{S}ol^s_g$ we may also assume $g_t$ to be analytic.
Then the higher derivatives of the curve nessecarily satisfy \eqref{recursion0}.
\begin{defn}
 Let $(M,g)$ be a shrinking Ricci soliton and $h=:g^{(1)}\in ISD$. We call $h$ integrable up to order $k$, if there exists a sequence of tensors $g^{(2)},\ldots,g^{(k)}\in C^{\infty}(S^2M)$ satisfying the recursion formulas $\eqref{recursion0}$.
If $h$ is integrable up to order $k$ for any $k\in\N$, $h$ is integrable of infinite order.
\end{defn}
\begin{lem}
 A tensor $h\in ISD$ is integrable in the sense of Definition \ref{DefnISD} if and only if it is integrable of infinite order.
\end{lem}
\begin{proof}
 If $h$ is integrable, it is obviously integrable of infinite order. Conversely, suppose we have $\tilde{g}^{(k)}\in C^{\infty}(S^2M)$, $k\in\N$, solving $\eqref{recursion0}$ for any $k$.
 Then for any $k\in\N$, we can construct a curve $\tilde{g}_t^k$ of metrics such that $\frac{d^l}{dt^l}|_{t=0}\tilde{g}^k_t=\tilde{g}^{(l)}$ and $\frac{d^l}{dt^l}\Sol_{\tilde{g}_t}=0$ for $l\in\left\{1,\ldots,k\right\}$ .
 By projecting these curves to the subset $\mathcal{Z}^s$ in the slice $\mathcal{S}_f^s$ suitably,
we have curves $g_t^k$ in $\mathcal{Z}^s$ and a sequence $g^{(k)}\in C^{\infty}(S^2M)$ such that for each $k\in\N$, we have $\frac{d^l}{dt^l}|_{t=0}g^k_t=g^{(l)}$ and $\frac{d^l}{dt^l}\Sol_{g_t}=0$ for $l\in\left\{1,\ldots,k\right\}$
and $g^{(k)}$ satisfies \eqref{recursion0}. Therefore, 
since $\mathcal{S}ol^s_g\subset \mathcal{Z}^s$ is an analytic subset, we can apply \cite[Theorem (1.2)]{Art68}:
The existence of the formal solution $g+\sum_{k=1}^{\infty}\frac{t^k}{k!}g^{(k)}$ of the equation $\Sol=0$ implies the existence of a real solution $g_t$ of $\Sol=0$ in $\mathcal{Z}^s$ such that $\frac{d}{dt}|_{t=0}g_t=g^{(1)}=h$.
Thus, $h$ is integrable.
\end{proof}

\begin{lem}\label{rigidity}
 Let $(M,g)$ be a shrinking Ricci soliton. If all $h\in ISD$ are integrable only up to some finite order, then $g$ is rigid.
\end{lem}
\begin{proof}
 If $g$ was not isolated, the analyticity of $\mathcal{S}ol^s_g$ implies the existence of a smooth curve $g_t\subset\mathcal{S}ol^s_g$ through $g$ and $h=\frac{d}{dt}|_{t=0}g_t\in ISD$ is integrable of order infinity.
\end{proof}
By the second variation of $\nu$, $D\Sol=-\tau N$ where $N$ is the stability operator. Thus the operator
$D\Sol:C^{\infty}(S^2M)\to C^{\infty}(S^2M)$ is formally self-adjoint with respect to the $L^2(e^{-f_g}dV)$-scalar product. Provided we already found $g^{(1)},\ldots g^{(k-1)}$ recursively via \eqref{recursion0},
the equation \eqref{recursion0} for $g^{(k)}$ has a solution if and only if
\begin{align*}
 \sum_{l=2}^{k}\sum_{\substack{ 1\leq k_1\leq\ldots\leq k_l\\ k_1+\ldots+k_l=k}} C(k,l,k_1,\ldots,k_l) D^l\Sol(g^{(k_1)},\ldots,g^{(k_l)})\in \kernel(D\Sol)^{\perp}.
\end{align*}
In the above line and in the next lemma, the orthogonal complement is taken with respect to $L^2(e^{-f_g}dV)$.
\begin{lem}
 Let $(M,g)$ be a Ricci soliton and $h\in ISD$. Then $h$ is integrable up to second order if and only if $D^2\Sol(h,h)\perp ISD$. 
\end{lem}
\begin{proof}
 We claim that $D^2\Sol(h,h)\perp \R \cdot g\oplus \delta_f^*(\Omega^1(M))$. Let $\tilde{h}\in  \R \cdot g\oplus \delta_f^*(\Omega^1(M))$ and
let $a_t\varphi_t^*g$ be a curve of homothetic metrics such that $a_0=1,\varphi_0=\identity_M$ and $\frac{d}{dt}|_{t=0}a_t\varphi_t^*g=\tilde{h}$. Consider the two-parameter-family of metrics given by
$g(s,t)=a_t\varphi_t^*(g+sh)$. By the first variation of $\nu$,
\begin{align*}
 \frac{d}{dt}\nu(g(s,t))=-\frac{1}{(4\pi\tau)^{n/2}}\int_M\langle \Sol,\frac{d}{dt}g\rangle\edv.
\end{align*}
By scale and diffeomorphism invariance, $\nu(g(s,t))$ only depends on $s$. Thus, by differentiating the above twice and using $D\Sol(h)=0$, we have
\begin{align*}
 0=\frac{d^2}{ds^2}\frac{d}{dt}\bigg\vert_{t,s=0}\nu(g(s,t))=-\frac{1}{(4\pi\tau)^{n/2}}\int_M\langle D^2\Sol(h,h),\tilde{h}\rangle \edv,
\end{align*}
which proves the claim. We argued above that $h$ is integrable up to second order if and only if $D^2\Sol(h,h)$ is orthogonal to $\kernel(D\Sol)$ but we have the orthogonal decomposition
\begin{align*}
 \kernel(D\Sol)=ISD\oplus  \R \cdot g\oplus\delta_f^*(\Omega^1(M))
\end{align*}
and so the statement of the lemma follows from the claim.
\end{proof}
Now we consider an Einstein manifold $g$ with constant $\mu>0$ and suppose that $2\mu\in\spectrum(\Delta)$. For the rest of the section, we focus on infinitesimal solitonic deformations contained in
\begin{align*}
 \left\{\mu v\cdot g+\nabla^2 v|v\in C^{\infty}(M),\Delta v=2\mu v\right\}.
\end{align*}
By diffeomorphism invariance, we only have to deal with the conformal part of these deformations which makes the calculations easier.
\begin{lem}
 Let $(M,g)$ be an Einstein manifold $g$ with constant $\mu>0$ and let $v\in C^{\infty}(M)$ so that $\Delta v=2\mu v$. Then,
 \begin{align*}
  D^2\ric(v\cdot g,v\cdot g)=&-(\frac{n}{2}-2)|\nabla v|^2g-2\mu v^2 g+3(\frac{n}{2}-1)\nabla v\otimes\nabla v+(n-2)\nabla^2v\cdot v,\\
  D^2(\nabla^2 f_g)(v\cdot g,v\cdot g)=&\nabla^2 u-(n-2)\nabla v\otimes\nabla v+(\frac{n}{2}-1)|\nabla v|^2g,
 \end{align*}
where $u$ is a solution of the equation
\begin{align*}
 (\frac{1}{\mu}\Delta-1)u=D^2\tau(v\cdot g,v\cdot g)-\frac{1}{\mu}[n\mu v^2+(-\frac{3}{4}n+\frac{1}{2})|\nabla v|^2].
\end{align*}
\end{lem}
\begin{proof}
 The calculations were already done in \cite[pp.\ 22-24]{Kro13} while computing a third variation of the shrinker entropy.
The second Fr\'{e}chet derivative of the Ricci tensor was computed in \cite[p.\ 23]{Kro13}. Moreover, we have
\begin{align*}
 (\nabla^2 f)''=\nabla^2(f'')-\nabla f'\otimes\nabla v-\nabla v\otimes\nabla f'+\langle \nabla f',\nabla v\rangle g,
\end{align*}
where the primes denote Fr\'{e}chet derivatives in the direction of $v\cdot g$. The function $u:=f''$ satisfies the equation in the statement of the lemma and
$f'=(\frac{n}{2}-1)v$.
\end{proof}
\begin{lem}\label{D^2sol}
 Let $(M,g)$ be an Einstein manifold $g$ with constant $\mu>0$ and let $v\in C^{\infty}(M)$ so that $\Delta v=2\mu v$. Then,
 \begin{align*}
  D^2\Sol(v\cdot g,v\cdot g)=&\frac{1}{2\mu}(\nabla^2 u-2\mu v^2 g+|\nabla v|^2g+(\frac{n}{2}-1)\nabla v\otimes\nabla v+(n-2)\nabla^2v\cdot v)\\
                             &+D^2\tau(v\cdot g,v\cdot g)\mu g,
 \end{align*}
where $u$ is a solution of the equation
\begin{align*}
 (\frac{1}{\mu}\Delta-1)u=D^2\tau(v\cdot g,v\cdot g)-\frac{1}{\mu}[n\mu v^2+(-\frac{3}{4}n+\frac{1}{2})|\nabla v|^2].
\end{align*}
\end{lem}
\begin{proof}
 We have $\tau_g=\frac{1}{2\mu}$ and by \cite[Lemma 2.4]{CZ12}, $D\tau(v\cdot g)=0$ because $\int_M v\dv=0$. This yields
\begin{align*}
 D^2\Sol(v\cdot g,v\cdot g)=D^2\tau(v\cdot g,v\cdot g)\mu g+\frac{1}{2\mu}(D^2\ric(v\cdot g,v\cdot g)+D^2(\nabla^2 f_g)(v\cdot g,v\cdot g)).
\end{align*}
Now the result follows from the previous lemma.
\end{proof}
\begin{thm}\label{secondordercriterion}
 Let $(M,g)$ be an Einstein manifold with Einstein constant $\mu>0$. Let $v\in C^{\infty}(M)$ be such that $\Delta v=2\mu v$. Then $h=\mu v\cdot g+\nabla^2 v\in ISD$ is not integrable of second order if there exists another function
 $w\in C^{\infty}(M)$ with $\Delta w=2\mu w$ such that
\begin{align*}
 \int_M v^2 w \dv\neq 0.
\end{align*}
\end{thm}
\begin{proof}
Suppose the contrary. Then, there exists a smooth curve $g_t$ of metrics such that $\frac{d}{dt}|_{t=0}g_t=h$ and $\frac{d}{dt}|_{t=0}\Sol_{g_t}=\frac{d^2}{dt^2}|_{t=0}\Sol_{g_t}=0$. By pulling back by a suitable family of diffeomorphisms, we obtain a smooth curve
$\tilde{g}_t$ of Ricci solitons such that $\frac{d}{dt}|_{t=0}\tilde{g}_t=v\cdot g$ and $\frac{d}{dt}|_{t=0}\Sol_{\tilde{g}_t}=\frac{d^2}{dt^2}|_{t=0}\Sol_{\tilde{g}_t}=0$.
By the arguments at the beginning of this section, this implies that
\begin{align}\label{perp}
D^2\Sol(v\cdot g,v\cdot g)\in\kernel(D\Sol)^{\perp}.
\end{align}
On the other hand, pick $w\in C^{\infty}(M)$ be as in the statement of the theorem. Then, $D\Sol(w\cdot g)=0$ by the calculations in \cite[p.\ 22]{Kro13}. Furthermore, straightforward calculations, using the eigenvalue equation $\Delta w=2\mu w$ and Lemma \ref{D^2sol} show that
\begin{align*}
\int_M \langle D^2\Sol(v\cdot g,&v\cdot g),w\cdot g\rangle\dv=\frac{1}{2\mu}\int_M \langle\nabla^2 u-2\mu v^2 g+|\nabla v|^2g,w\cdot g\rangle\dv\\
                       &+\frac{1}{2\mu}\int_M\langle(\frac{n}{2}-1)\nabla v\otimes\nabla v+(n-2)\nabla^2v\cdot v,w\cdot g\rangle\dv\\
                       =&-\int_M u\cdot w\dv-2(n-1)\int_M v^2w\dv+\frac{1}{\mu}(\frac{3n}{4}-\frac{1}{2})\int_M |\nabla v|^2w\dv\\
                       =&\frac{1}{\mu}\int_M \left(\frac{1}{\mu}\Delta-1\right)^{-1}\left[D^2\tau(v\cdot g,v\cdot g)\mu+n\mu v^2+\left(-\frac{3}{4}n+\frac{1}{2}\right)|\nabla v|^2\right]\cdot w\dv\\
                        &-2(n-1)\int_M v^2w\dv+\frac{1}{\mu}(\frac{3n}{4}-\frac{1}{2})\int_M |\nabla v|^2w\dv\\
                       =&\frac{1}{\mu}\int_M \left[D^2\tau(v\cdot g,v\cdot g)\mu+n\mu v^2+\left(-\frac{3}{4}n+\frac{1}{2}\right)|\nabla v|^2\right]\cdot \left(\frac{1}{\mu}\Delta-1\right)^{-1}w\dv\\
                        &-2(n-1)\int_M v^2w\dv+\frac{1}{\mu}(\frac{3n}{4}-\frac{1}{2})\int_M |\nabla v|^2w\dv\\
                       =&-(n-2)\int_M v^2w\dv\neq0
\end{align*}
which contradicts \eqref{perp}. The terms containing $D^2\tau(v\cdot g,v\cdot g)$ vanish after integration because $w$ has vanishing integral.
\end{proof}
\begin{rem} Theorem \ref{secondordercriterion} and Lemma \ref{rigidity} imply Theorem \ref{rigiditytheorem}.
 \end{rem}
\section{Examples}\label{examples}
As an application of the above criterion, we are prove the following
\begin{thm}
 All infinitesimal solitonic deformations of $(\CP^{2n},g_{fs})$ are not integrable of second order.
\end{thm}
\begin{proof}Let $\mu$ be the Einstein constant. Recall that $(\CP^{m},g_{fs})$ has no infinitesimal Einstein deformations (see e.g.\ \cite[Table 2]{CH13}) but $2\mu\in\spectrum(\Delta)$ and thus, the space of infinitesimal solitonic deformations is given by
\begin{align*}
\left\{\mu v\cdot g+\nabla^2 v|v\in C^{\infty}(M),\Delta v=2\mu v\right\}.
\end{align*}
We will show that for any such $v$, there exists a function $w$ in the same eigenspace such that $\int_M v^2 w\dv\neq0$. Then Theorem \ref{secondordercriterion} yields the result.

Let us briefly recall the construction of eigenfunctions on $\CP^{2n}$ as in \cite[pp.\ 172-173]{BGM71}.
Consider $\C^{2n+1}=\R^{4n+2}$ with coordinates $(x_1,\ldots,x_{2n+1},y_1,\ldots,y_{2n+1})$ and let $z_j=x_j+iy_j$, $\bar{z}_j=x_j-iy_j$
be the complex coordinates\index{complex coordinates}. Defining\index{$\partial_{x_i}$, directional derivative} $\partial_{z_j}=\frac{1}{2}(\partial_{x_j}-i\partial_{y_j})$ and $\partial_{\bar{z}_j}=\frac{1}{2}(\partial_{x_j}-i\partial_{y_j})$,
we can rewrite the Laplace operator on $\C^{2n+1}$ as
\begin{align*}\Delta=-4\sum_{j=1}^{2n+1}\partial_{z_j}\circ \partial_{\bar{z}_j}.
\end{align*}
Let $P_{k,k}$ be the space of polynomials on $\C^{2n+1}$ which are homogeneous of degree $k$ in $z$ and $\bar{z}$ and let $H_{k,k}$ the subspace of harmonic polynomials in $P_{k,k}$. We have the decomposition
$P_{k,k}=H_{k,k}\oplus r^2P_{k-1,k-1}$.
Elements in $P_{k,k}$ are $S^1$-invariant and thus, they descend to functions on the quotient $\CP^{2n}=S^{2n+1}/S^1$. The eigenfunctions on $\CP^{2n}$ correspond to functions in $H_{k,k}$, $k\geq0$. Let $v$ be an eigenfunction to the eigenvalue $2\mu$. Then $v$ corresponds to a function $f\in H_{1,1}$. We may assume that $f$ is of the form
$f(z,\bar{z})=\sum_{i}\lambda_i |z_i|^2$ and $\sum_{i} \lambda_i=0$.
Consider its square
\begin{align*}
f^2\in P_{2,2}=H_{2,2}\oplus r^2 H_{1,1}\oplus \R \cdot r^4
\end{align*}
To prove the claim, it suffices to show that $f^2\notin H_{2,2}\oplus\R \cdot r^4$ or equivalently, $\Delta(f^2)\notin\R\cdot \Delta (r^4)$. In fact, we have
\begin{align*}\Delta(f^2)=&\Delta\left(\sum_i\lambda_i^2 |z_i|^4\right)+\Delta\left(
\sum_{i\neq j}\lambda_i\lambda_j |z_i|^2|z_j|^2\right)\\
=&-16\sum_i\lambda_i^2 |z_i|^2-4\sum_{i\neq j}\lambda_i\lambda_j(|z_i|^2+|z_j|^2)
=-8\sum_i\lambda_i^2 |z_i|^2.
\end{align*}
On the other hand,
\begin{align*}
\Delta(r^4)=&\Delta\left( |z_i|^4\right)+\Delta\left(
\sum_{i\neq j} |z_i|^2|z_j|^2\right)=-16\sum_i |z_i|^2-4\sum_{i\neq j}(|z_i|^2+|z_j|^2)=-16(n+1)r^2.
\end{align*}
Thus, if $\Delta(f^2)\in\C\cdot \Delta (r^4)$, $|\lambda_i|$ is independant of $i$ but this contradicts  $\sum_{i}\lambda_i=0$.

\end{proof}
\begin{rem}
 Together with Lemma \ref{rigidity}, the above result implies Theorem \ref{CP2n}.
\end{rem}

\begin{rem}
 The proof of the above theorem shows the following: The $\CP^{2n+1}$ have a zero set of infinitesimal solitonic deformations (in the space of all ISD's) which are integrable up to second order.
For these deformations, $D^2\Sol(h,h)$ is orthogonal to all conformal deformations of the form $w\cdot g$, $w\in E(2\mu)$ and the proof of Lemma \ref{secondordercriterion} implies that
$D^2(h,h)$ is orthogonal to all ISD's. 
The statement then follows from Lemma \ref{secondordercriterion}.
 However, it is not clear whether these deformations are integrable up to higher order. It seems likely that also $\CP^{2n+1}$ is rigid.
\end{rem}
\begin{exmp}Consider $S^2\times S^2$ with the product of round metrics and let $\mu$ be the Einstein constant of the metric. Then, $2\mu$ is in the spectrum of the product metric and there are no infinitesimal Einstein deformations.
The space of infinitesimal solitonic deformations is therefore again equal to
\begin{align*}
\left\{\mu v\cdot g+\nabla^2 v|v\in C^{\infty}(M),\Delta v=2\mu v\right\}.
\end{align*}
Moreover, $2\mu$ is the smallest eigenvalue of Laplacian on the product metric and its eigenspace is
\begin{align*}
 E(2\mu)=\left\{\pr_1^*v+\pr_2^*w|v,w\in C^{\infty}(S^2), \Delta v=2\mu v,\Delta w=2\mu w\right\}.
\end{align*}
The eigenfunctions on the factors are restrictions of linear functions on $\R^3$ (e.g.\ \cite[Chapter III C.I.]{BGM71}) and thus, they are antisymmetric with respect to the antipodal map $\sigma\in\Iso(S^2,g_{st})$.
Therefore, $(\sigma\times\sigma)^*v=-v$ for any $v\in E(2\mu)$ and so, $\int_{S^2\times S^2}v^2 w\dv=0$ for any $v,w\in E(2\mu)$. 
The same argumentation as in the remark above implies that all infinitesimal solitonic deformations on $S^2\times S^2$ are integrable up to second order. 
It is again not clear whether they are integrable of higher order.
\end{exmp}
\begin{rem}
 In \cite{Koi82}, Koiso showed that the product Einstein metric on $S^2\times \CP^{2n}$ is rigid as an Einstein metric although it has infinitesimal Einstein deformations. In fact, the infinitesimal Einstein deformations are
 linear combinations of the form $\alpha v\cdot g_1+\beta v\cdot g_2+\gamma \nabla^2 v$ where $\alpha,\beta,\gamma\in\R$, $v=pr_2^*w$ and $w\in E(2\mu)$ and they are not integrable of second order.
 
It is not clear if $S^2\times \CP^{2n}$ is rigid as a soliton. The eigenfunctions to the first nonzero eigenvalue on the first factor form infinitesimal solitonic deformations which are integrable up to second order.
\end{rem}

\vspace{3mm}

\textbf{Acknowledgement.} The author would like to thank the
Sonderforschungsbereich 647
funded by the
Deutsche Forschungsgemeinschaft
for financial support.
\newcommand{\etalchar}[1]{$^{#1}$}


\begin{thebibliography}{CCG{\etalchar{+}}07}
 
\providecommand{\url}[1]{\texttt{#1}}
\expandafter\ifx\csname urlstyle\endcsname\relax
  \providecommand{\doi}[1]{doi: #1}\else
  \providecommand{\doi}{doi: \begingroup \urlstyle{rm}\Url}\fi

\bibitem[Art68]{Art68}
\textsc{Artin}, Michael:
\newblock On the Solutions of Analytic Equations
\newblock {In: }\emph{Invent. Math. } \textbf{5} (1968), no. 4, 277--291

\bibitem[BGM71]{BGM71}
\textsc{Berger}, Marcel ; \textsc{Gauduchon}, Paul  ; \textsc{Mazet}, Edmond:
\newblock \emph{{Le spectre d'une vari\'et\'e riemannienne.}}
\newblock Lecture Notes in Mathematics, 194, Springer-Verlag, Berlin-Heidelberg, 1971

\bibitem[Bes08]{Bes08}
\textsc{Besse}, Arthur~L.:
\newblock \emph{{Einstein manifolds. Reprint of the 1987 edition.}}
\newblock {Springer-Verlag, Berlin}, 2008

\bibitem[Cao10]{Cao10}
\textsc{Cao}, Huai-Dong:
\newblock \emph{Recent progress on Ricci solitons. Recent advances in geometric analysis}, 1–38, Adv. Lect. Math., 11, 
\newblock Int. Press, Somerville, MA,
\newblock 2010

\bibitem[CHI04]{CHI04}
\textsc{Cao}, Huai-Song ; \textsc{Hamilton}, Richard  ; \textsc{Ilmanen}, Tom:
\newblock Gaussian densities and stability for some {R}icci solitons.
\newblock arXiv preprint math/0404165
\newblock   (2004).

\bibitem[CH13]{CH13}
\textsc{Cao}, Huai-Dong ; \textsc{He}, Chenxu:
\newblock Linear {S}tability of {P}erelmans $\nu$-entropy on {S}ymmetric spaces
  of compact type.
\newblock {In: } \emph{J. Reine Angew. Math. } (2013), 1--18

\bibitem[CZ12]{CZ12}
\textsc{Cao}, Huai-Dong ; \textsc{Zhu}, Meng:
\newblock {On second variation of Perelman's Ricci shrinker entropy.}
\newblock {In: }\emph{Math.\ Ann.\ } \textbf{353} (2012), no. 3, 747--763

\bibitem[CCG{\etalchar{+}}07]{CC07}
\textsc{Chow}, Bennett ; \textsc{Chu}, Sun-Chin ; \textsc{Glickenstein}, David
  ; \textsc{Guenther}, Christine ; \textsc{Isenberg}, James ; \textsc{Ivey},
  Tom ; \textsc{Knopf}, Dan ; \textsc{Lu}, Peng ; \textsc{Luo}, Feng  ;
  \textsc{Ni}, Lei:
\newblock \emph{{The Ricci flow: techniques and applications. Part I: Geometric
  aspects.}}
\newblock Mathematical Surveys and Monographs 135, American
  Mathematical Society, Providence, RI, 2007

\bibitem[Ebi70]{Eb70}
\textsc{Ebin}, David~G.:
\newblock {The manifold of Riemannian metrics.}
\newblock {In: }\emph{Proc. {S}ymp. {AMS}} Bd.~15, 1970, 11--40

\bibitem[Ham88]{Ham88}
\textsc{Hamilton}, Richard~S.:
\newblock {The Ricci flow on surfaces.}
\newblock {In: }\emph{Contemp.\ Math.\ } \textbf{71} (1988), no. 1

\bibitem[Koi78]{Koi78}
\textsc{Koiso}, Norihito:
\newblock {Non-deformability of Einstein metrics.}
\newblock {In: }\emph{Osaka J.\ Math.\ } \textbf{15} (1978), no. 2, 419--433

\bibitem[Koi80]{Koi80}
\textsc{Koiso}, Norihito:
\newblock {Rigidity and stability of Einstein metrics - The case of compact
  symmetric spaces.}
\newblock {In: }\emph{Osaka J.\ Math.\ } \textbf{17} (1980), 51--73

\bibitem[Koi82]{Koi82}
\textsc{Koiso}, Norihito:
\newblock {Rigidity and infinitesimal deformability of Einstein metrics.}
\newblock {In: }\emph{Osaka J.\ Math.\ } \textbf{19} (1982), no. 3, 643--668

\bibitem[Koi83]{Koi83}
\textsc{Koiso}, Norihito:
\newblock {Einstein metrics and complex structures.}
\newblock {In: }\emph{Invent.\ Math.\ } \textbf{73} (1983), no. 1, 71--106

\bibitem[Kr\"o13]{Kro13}
\textsc{Kr\"oncke}, Klaus:
\newblock {Einstein metrics, Ricci flow and the Yamabe invariant.}
\newblock arXiv preprint arXiv:1312.2224,
\newblock   (2013).

\bibitem[Kr\"o14]{Kro14}
\textsc{Kr\"oncke}, Klaus:
\newblock {Stability and Instability of Ricci solitons}
\newblock {In: }\emph{Calc. Var. Partial Differ. Eqn.} (2014), 1--23

\bibitem[Per02]{Per02}
\textsc{Perelman}, Grisha:
\newblock {The entropy formula for the Ricci flow and its geometric
  applications.}
\newblock arXiv preprint math/0211159,
\newblock   (2002).

\bibitem[PS13]{PS13}
 \textsc{Podest{\`a}}, Fabio ; \textsc{Spiro}, Andrea:
 \newblock {On moduli spaces of Ricci solitons.}
 \newblock {In: }\emph{J.\ Geom.\ Anal.\ } \textbf{25} (2015), no. 2, 1157-1174

\end{thebibliography}
\end{document}